\documentclass[12pt]{article}

\usepackage{amsthm}
\usepackage{amsmath,amssymb}
\usepackage{verbatim}
\usepackage{color}
\usepackage{graphicx}

\newtheorem{theorem}{Theorem}[section]
\newtheorem{thm}[theorem]{Theorem}

\newtheorem{lem}[theorem]{Lemma}

\numberwithin{equation}{section}

\def\be{\begin{equation}}
\def\ee{\end{equation}}
\def\bes{\begin{equation*}}
\def\ees{\end{equation*}}

\usepackage{amsmath,amssymb}
\usepackage{authblk}

\renewcommand{\P}{\mathbb{P}}
\newcommand{\E}{\mathbb{E}}
\newcommand{\Rd}{\mathbb{R}^d}

\newcommand{\Rp}{\mathbb{R}^+}

\newcommand{\barr}{\begin{array}{rcl}}
\newcommand{\earr}{\end{array}}

\newcommand{\disp}{\displaystyle}
\newcommand{\Ema}{\E^{m_a}}
\newcommand{\Pma}{\P^{m_a}}
\newcommand{\ind}{1{\hskip -2.5 pt}\hbox{\textnormal I}}

\begin{document}

\title{An Extension of a Boundedness Result for Singular Integral Operators}
%\subtitle{Do you have a subtitle?\\ If so, write it here}

%\titlerunning{A Boundedness Result}        % if too long for running head

\author{Deniz Karl\i}
\affil{Department of Mathematics\\ AMF233, I\c{s}\i k University\\
 34980 \c{S}ile, Istanbul, Turkey\\
		{deniz.karli@isikun.edu.tr} \& deniz.karli@gmail.com}

%\authorrunning{Short form of author list} % if too long for running head

%\institute{   I\c{s}\i k University   }

\date{}
% The correct dates will be entered by the editor

\maketitle

\maketitle

\renewcommand{\thefootnote}{}
\footnote{This research project is supported by the BAP grant numbered 14B103 at the I\c{s}\i k University, Istanbul, Turkey.}

\section{Introduction and Preliminaries}

Boundedness of singular integral operators has been studied for a long time. There are some well-known results  which were proved first by using classic analytical techniques. In these techniques, there are some important operators providing intermediate steps for the proof. Three often used operators are Lusin Area functional ($A_f$), non-tangential maximal function ($N_\alpha^f$) and G-star functional ($G^*_{f}$). They played an important role in the development of Harmonic Analysis. (See Stein \cite{stein} and \cite{stein3}.)

 With the introduction of probabilistic techniques, alternative proofs have come to the surface in addition to these analytical tools. In these classical techniques, Brownian motion plays a central role. One such approach is to consider a $(d+1)$-dimensional Brownian motion on the upper half-space and provide a probabilistic definition of harmonic functions in terms of martingales. By means of martingales, one can define Littlewood-Paley functions and hence provide probabilistic proofs of boundedness of some operators. (See, for example, Varopoulos \cite{varopoulos}, Burkholder and Gundy \cite{BG}, Burkholder, Gundy and Silverstein \cite{BGS}, Durrett \cite{durrett} and Bass \cite{bass_prob_tech}. For a more detailed literature study on square functions and these operators, see Ba$\tilde{\mbox{n}}$uelos and Davis \cite{banuelos2}.)

In the paper \cite{karli}, we studied a more general process in $(d+1)$-dimensional half-space $\Rd\times \Rp$. We would like to obtain generalisations of some theorems using the power of probabilistic techniques and the weaker conditions imposed by the process which we start with. This paper can be considered as a continuation of the discussion which originates from \cite{karli}.

The main results of this paper include (i.) boundedness of two important operators, namely the Area functional and the $G^*$ functional, and (ii.) an extension of a classical multiplier theorem on singular integrals. This classical version of the multiplier theorem, which we will discuss here, focuses on singular integrals with kernels $\kappa:\Rd\rightarrow \mathbb{R}$ satisfying the cancelation property
\begin{align}\label{cancelation} 
	\int_{r<|x|<R} \kappa(x)dx=0, \qquad \mbox{for all}\quad 0<r<R.
\end{align}
Together with a smoothing condition and some control on its tail, it is known that the corresponding convolution operator is bounded. The classical version is stated as follows. (The proof of the case $d=1$ is given in  \cite[Theorem 5.3, P.270]{bass_prob_tech}. For $d>1$, the same argument applies easily with a slight modification. See also \cite[Theorem 1.1]{bass_sing_int}.)
\begin{thm}\label{classic_case}
 Suppose $\kappa$ is the kernel of a convolution operator $T$. If $\kappa\in\mathcal{C}^1$, it satisfies the cancelation condition (\ref{cancelation}), 
 \begin{align}\label{condition}
 	|\kappa(x)|\leq c |x|^{-d} \mbox{ and } |\nabla \kappa(x)|\leq c |x|^{-d-1} ,\quad x\not= 0
\end{align}
 then for any $1<p<\infty$ there is a finite constant $c_p$ depending only on $p$ such that
 	$$\|T\|_{L^p(\Rd)\rightarrow L^p(\Rd)} <c_p.$$
\end{thm}
Our goal is to weaken the condition (\ref{condition}) by replacing $d$ in the exponent with $d-1+\alpha/2$ for some $\alpha\in (1,2)$ when $|x|>1$ (Theorem \ref{main_thm}). We note that for $\alpha=2$, we obtain the condition (\ref{condition}).

First we introduce our notation and some preliminary results in this section. Throughout the paper $c$ will denote a positive constant. Its value may change from line to line. 

We consider a $d$-dimensional right continuous rotationally symmetric $\alpha$-stable process $\left(Y_t\right)_{t\geq 0}$ for $\alpha\in (0,2)$, that is, $\left(Y_t\right)_{t\geq 0}$ is a right continuous Markov process with independent and stationary increments whose characteristic function is $\E(e^{i\xi Y_s})=e^{-s|\xi|^\alpha}, \xi\in\Rd, s>0$. By $p(s,x,y)$, we will denote its  (symmetric) transition density such that 
$$\P^x(Y_s\in A)=\int_A p(s,x,y)\,dy,$$
and by $P_s$ we will denote the corresponding semi-group $P_s(f)(x)=\E^x(f(Y_s))$. Here $\P^x$ is the probability measure for the process started at $x\in \Rd$, and $\E^x$ is the expectation taken with respect to $\P^x$. The transition density $p(s,x,0)$ satisfies the scaling property 

\begin{align}\label{scaling} p(s,x,0)=s^{-d/\alpha}p(1,x/s^{1/\alpha},0), \qquad x\in\Rd,\, s>0.\end{align} 

Similarly, we denote  a one-dimensional Brownian motion (independent from $Y_s$) by $Z_s$ and the probability measure for the process started at $t>0$ by $\P^t$. The process of interest is the product $X_s=(Y_s,Z_s)$ started at $(x,t)\in\Rd\times\Rp$, the corresponding probability measure and the expectation are $\P^{(x,t)}$ and $\E^{(x,t)}$, respectively. Define the stopping time $T_0=\inf\{s\geq 0: Z_s=0\}$ which is the first time $X_t$ hits the boundary of $\Rd\times\Rp$. It is clear that $T_0$ and the process $Y$ are independent since $T_0$ is expressed in terms of $Z$ only.

To provide a connection between probabilistic and deterministic integrals, we will use two tools; a new measure $\Pma$ and the vertical Green function.  Denoting the Lebesgue measure on $\Rd$ by $m(\cdot)$, we define the measure $\Pma$ by
$$\Pma=\int_{\Rd} \P^{(x,a)} m(dx),\qquad a>0.$$
Let $\Ema$ denote the expectation with respect to this measure. We note that the law of $X_{T_0}$ under this measure is $m(\cdot)$. Moreover, the semi-group $P_t$ is invariant under the Lebesgue measure, that is,
\begin{align}\label{invariance}
	\int_{\Rd} P_tf(x)m(dx)=\int_{\Rd} f(x)m(dx).
\end{align}
This follows from the symmetry of the kernel and the conservativeness of $Y$.

Second, for a positive Borel function $f$, the vertical Green function, which is the Green function for one-dimensional Brownian motion, is given by
\begin{align}\label{greens_function}
	\E^a\left[\int_0^{T_0}f(Z_s)\, ds\right]=\int_0^\infty (s\wedge a)f(s)ds.
\end{align}

 Harmonic functions play a key role in showing boundedness of Littlewood-Paley operators. Here we adapt the probabilistic interpretation of  a harmonic function (with respect to the process $X$). A continuous function $u:\Rd\times\Rp\rightarrow \mathbb{R}$ is said to be harmonic (or $\alpha$-harmonic) if $u(X_{s\wedge T_0})$ is a martingale with respect to the filtration $\mathcal{F}_s=\sigma(X_{r\wedge T_0}: r\leq s)$ and the probability measure $\P^{(x,t)}$ for any starting point $(x,t)\in\Rd\times\Rp$. One way to obtain such a harmonic function is to start with a bounded Borel function $f:\Rd\rightarrow\mathbb{R}$ and define its extension $u$ by
\begin{align*}
	u(x,t):=\E^{(x,t)}f(Y_{T_0})=\int_0^\infty \E^xf(Y_s)\P^t(T_0\in ds),
\end{align*}
where $\P^t(T_0\in ds)$ is the exit distribution of one-dimensional Brownian motion from $(0,\infty)$ which is given by
\begin{align*}
	\mu_t(ds):=\P^t(T_0\in ds)=\frac{t}{2\sqrt{\pi}}e^{-t^2/4s}s^{-3/2}ds
\end{align*}
(see \cite{meyer_long}). By a slight abuse of notation, we will denote both the function on $\Rd$ and its extension to the upper-half space by the same letter, that is, $f_t(x):=f(x,t)=\E^{(x,t)}f(Y_{T_0})$. Next, we define the semi-group $Q_t=\int_0^\infty P_s \mu_t(ds)$.  This semi-group provides us a representation of the extension 
\begin{align*}
	\disp f_t(x)=f(x,t)=Q_tf(x)=\int_{\Rd} f(y)\int_0^\infty p(s,x,y)\mu_t(ds)dy.
\end{align*}
We note that this is a convolution with the probability kernel $$q_t(x)=\int_0^\infty p(s,x,0)\mu_t(ds),$$ whose Fourier transform  is $e^{-t|\cdot|^{\alpha/2}}$. So $q_t(x)$ can be identified with the density of a symmetric $\alpha/2$-stable process, which will allow us to write the estimate (\ref{qt_estimate}) below. Moreover, $q_t(x)$ is radially decreasing in $x$. To see this, it is enough to write the representation
\begin{align}\label{dens_rep}
		p(1,x,0)=\int_{\Rd}\frac{1}{(4\pi s)^{d/2}} e^{-|x|^2/(4s)}g_{\alpha/2}(1,s)ds,
	\end{align}
where $g_{\alpha/2}$ is the density of an $\alpha/2$ stable subordinator whose Laplace transform is given by $\int_0^\infty e^{-\lambda v} g(s,v)dv=e^{-s\lambda^{\alpha/2}}$. (See \cite[p. 261]{sato}  for details.) 

One of the key tools in proving certain inequalities is the density estimates on $p(s,x,0)$. Although there is an infinite series expansion, it is not very easy to work with. For this purpose, we will use a well-known two-sided estimate
  \begin{align}\label{sas_estimate}
 		c_1\,(s^{-d/\alpha} \wedge \frac{s}{|x-y|^{d+\alpha}}) \leq p(s,x,y) \leq c_2\,(s^{-d/\alpha} \wedge \frac{s}{|x-y|^{d+\alpha}}), 
  \end{align}
$(s,x,y)\in\Rp\times\Rd\times\Rd$, which allows us to control the tail of the transition density. (See \cite[Theorem 2.1]{blumenthal}.) 
This estimate leads to an estimate on $q_t(x)$ due to the observation that it coincides with the density of a symmetric $\alpha/2$-stable process. We have 
  \begin{align}\label{qt_estimate}
 		c_1\,(t^{-2d/\alpha} \wedge \frac{t}{|x|^{d+\frac{\alpha}{2}}}) \leq q_t(x) \leq c_2\,(t^{-2d/\alpha} \wedge \frac{t}{|x|^{d+\frac{\alpha}{2}}}). 
  \end{align}

In addition, we will need to control the derivative of $p(s,x,0)$. The following Lemma provides this control. Let $\partial^k_{x_j}$ denote the $k^{th}$ partial derivative in the direction of $j^{th}$ coordinate.

\begin{lem}\label{lemma_der_of_dens}
For $k=1,2$ and $j=1,...,d$, we have
\begin{itemize}
	\item[i.] $\disp \left|\partial_{x_j}^kp(1,x,0)\right|\leq c \left( 1 \wedge \frac{1}{|x|^k}\right)p(1,x,0)$ and \\
	\item[ii.]  $\disp \left|\partial_{x_j}^kp(t,x,0)\right|\leq c \left( t^{-k/\alpha} \wedge \frac{1}{|x|^k}\right)p(t,x,0)$ whenever $t>0$.
\end{itemize}
\end{lem}
This Lemma is a direct consequence of Proposition 3.3 of \cite{bass_zqchen} and the inequality (\ref{sas_estimate}) above. 

For the rest of the paper, we will need some results and definitions from \cite{karli}. To keep this paper as much self-contained as possible, we provide some of these theorems and definitions here. For details, we refer to \cite{karli}.
One of the main results of \cite{karli} is that harmonic functions, defined above, satisfy the Harnack inequality. We will use this result to show boundedness of some operators in the next section. Let us denote by $D_r$ the open rectangular box with center $(y,s)\in\Rd\times\Rp$
$${D}_r=\{ (x,t)\in\Rd\times\Rp: |x_i-y_i|<\frac{r^{2/\alpha}}{2} ,\,i=1,...,d,\, x=(x_1,...,x_d),|s-t|<\frac{r}{2}\}.$$ 
When using these rectangular boxes, we will consider nested boxes with the same center. That is why we don't include the center point in the notation for simplicity, and just write $D_r$ for these rectangular boxes.\\

\begin{thm}[K. '11]\label{harnack_inequality}
	There exists $c>0$ such that if $u$ is non-negative and bounded on $\mathbb{R}^d\times \mathbb{R}^+$, harmonic in ${D}_{16}$ and ${D}_{32}\subset\Rd\times\Rp$, then 
	\begin{align*}
			u(x,t)\leq c\, u(x',t')\,, & \quad (x,t),\,(x',t')\in {D}_1.
	\end{align*}
\end{thm}

Using this inequality, we proved a Littlewood-Paley Theorem. We defined a new operator with respect to our product process $X_s=(Y_s,Z_s)$. The horizontal component of the classical operator is replaced by the one corresponding to the symmetric stable process. The two components are defined as

$$
\disp \overrightarrow{G}_f(x) =\disp\left[ \int_0^\infty t\,\int_{\Rd} \frac{[f_t(x+h)-f_t(x)]^2}{|h|^{d+\alpha}} dh  \,dt  \right]^{1/2} \, ,
$$
and
\begin{align*}
\disp G^{\uparrow}_f(x) & =\disp\left[ \int_0^\infty t\, \left[ \frac{\partial}{\partial t} f(x,t) \right]^2 \,dt  \right]^{1/2},
\end{align*}
and hence the Littlewood-Paley operator $G_f$ is defined as 
\begin{align*}
              \disp G_f=\left[(\overrightarrow{G}_f)^2+ (G^{\uparrow}_f)^2\right]^{1/2}.
\end{align*}
Unlike the Brownian motion case, the Littlewood-Paley Theorem (Theorem \ref{LP_thm} part (i.))  cannot be extended to $p\in (1,2)$. This problem seems to occur due to the large jump terms of the horizontal process. That is why we truncated the part of the horizontal component which correspond to the large jumps. We denote this new operator obtained after truncation by $\overrightarrow{G}_{f,\alpha}$,
\begin{align*}
	\overrightarrow{G}_{f,\alpha}(x)  &  =\disp\left[ \int_0^\infty t\, \Gamma_\alpha(f_t,f_t)(x) \,dt  \right]^{1/2}, 
\end{align*}
where 
\begin{align}\label{gamma_alpha}
	\Gamma_\alpha(f_t,f_t)(x)=\int_{|h|<t^{2/\alpha}} [f_t(x+h)-f_t(x)]^2\frac{dh}{|h|^{d+\alpha}} ,
\end{align}\\
and the new restricted Littlewood-Paley operator is
\begin{align*}
	G_{f,\alpha}(x)=\left[ \left(\overrightarrow{G}_{f,\alpha}(x) \right)^2 + \left(G^{\uparrow}_f(x)\right)^2\right]^{1/2}.
\end{align*}

 \begin{thm}\label{LP_thm}
If $f\in L^p(\Rd)$, then for some constant $c>0$
	\begin{itemize}
		\item[i.] $\|G_f\|_p\leq c \|f\|_p$ for $p\geq 2$,\\
		\item[ii.]  $\|G^{\uparrow}_f\|_p\leq c \|f\|_p$ for $p>1$ and\\
		\item[iii.] $\|\overrightarrow{G}_{f,\alpha}\|_p\leq c \|f\|_p$ for $p>1$.
	\end{itemize}
\end{thm}

Part (i.) is due to a work by P.A. Meyer \cite{meyer_long}. This result is a special case of his work in which he studied symmetric Markov processes. Part (ii.) is studied by E.M. Stein in  \cite[Chapter V]{Stein2} in the case of symmetric semigroups. The proof of the third part is given in the paper \cite[Theorem 7]{karli}.

There are also some recent results based on an analytic approach to a differential equation where the fractional Laplacian is involved. In \cite{kim}, I. Kim and K. Kim discussed another operator by applying the fractional Laplacian to  $P_tf(x)$ where $P_t$ is defined as above. This operator plays the role of the classical Littlewood-Paley operator, where the Laplacian is the generator when $\alpha=2$ (that is, when the process is a Brownian motion) and hence the authors obtain an analogue of the classical inequality in fractional Laplacian case. However, as in Meyer's result (part i. of Theorem \ref{LP_thm}), this inequality holds for $p\geq 2$. One of our main results in the paper \cite{karli} (part iii. of Theorem \ref{LP_thm}) allows us to generalize this inequality first by considering the harmonic extension $Q_tf$ and then writing the integrand as the singular integral (\ref{gamma_alpha})  instead of the differential $\partial^\alpha_x$ on a restricted domain to provide some control over the large jump terms. Without this restriction, it is not possible to extend this result to $p\in(1,2)$. In this paper, we will make use of this inequality for $p>1$. 

In addition to the Theorem above, it is also not difficult to see that part (ii) can be written as a two sided-inequality. Here we provide a short proof by a well-known duality argument.
\begin{lem}\label{LP_lemma}
	If  $p>1$ and $f\in L^p(\Rd)\cap L^2({\Rd})$ then $\|f\|_p\leq c \, \|G^{\uparrow}_f\|_p$.
\end{lem}
\begin{proof}
	First note that by the Plancherel identity,
		\begin{align}\label{l2_eq}
			\|G^{\uparrow}_f\|_2^2&=c \int_0^\infty t \int_{\Rd} \left| \widehat{f}(\xi)\right|^2 |\xi|^\alpha e^{-2t|\xi|^{\alpha/2}} d\xi \, dt=c \|f\|_2,
		\end{align}
	since $\left(Q_tf\right)^{\widehat{}}(\cdot)=e^{-t|\cdot|^{\alpha/2}}\widehat{f}(\cdot)$. 
	
	Second, if $h\in L^q(\Rd)\cap L^2(\Rd)$, where $1/p+1/q=1$, then using polarization identity and equality (\ref{l2_eq}),
	\begin{align*}
		\int_{\Rd} f(x)h(x)dx&=\frac{1}{4}\left( \|f+h\|^2_2-\|f-h\|^2_2\right)\\
					& = c\left( \|G^{\uparrow}_{f+h}\|^2_2-\|G^{\uparrow}_{f-h}\|^2_2\right)\\ \\
					& = c \int_{\Rd} \int_0^\infty t \frac{\partial f}{\partial t}(x,t)\frac{\partial h}{\partial t}(x,t) dt dx. \\
	\end{align*}
	Using the Cauchy-Schwartz inequality and then the H\"older inequality, we obtain
	\begin{align*}
		\int_{\Rd} f(x)h(x)dx&\leq c \int_{\Rd} G^{\uparrow}_f(x) G^{\uparrow}_h(x) dx \\
					& \leq c \|G^{\uparrow}_f\|_p \|G^{\uparrow}_h\|_q \\
					&\leq c \|G^{\uparrow}_f\|_p \|h\|_q
	\end{align*}
	where the last inequality follows from Theorem \ref{LP_thm}. 
	
	Finally, the result follows if we take supremum over all such $h$ with $\|h\|_q\leq 1$. 
\end{proof}

In the classical Littlewood-Paley theory, there are some operators which are often used to prove  intermediate steps of boundedness arguments. We believe that they should be studied and analogous results with the classical theory should be provided in order to obtain a complete picture. In the next section we will discuss some of these operators and prove their boundedness in $L^p(\Rd)$. Among  these operators, two important ones are the Area functional and $G^*$ functional. The Area functional in our setup is given by
\begin{align*}
	A_f(x)=\left[\int_0^\infty \int_{|y|<t^{2/\alpha}} t^{1-2d/\alpha} \Gamma_\alpha(f_t,f_t)(x-y)dy\,dt\right]^{1/2}.
\end{align*}
The reason for this name is that it represents the area of $f(D)$ in the classical setup ($\alpha=2$ and $\Gamma_\alpha$ is replaced by $|\nabla|^2$) where $D$ is the cone $\{(y,t): |y-x|<t\}$ and $d=2$. 

Second, we define the new $G^*$ functional by means of its horizontal and vertical components. But first we denote by $K_t^\lambda$ the function 
\begin{align*}
	K_t^\lambda(x)=t^{-2d/\alpha}\left[ \frac{t^{2/\alpha}}{t^{2/\alpha}+|x|}\right]^{\lambda d},\qquad t>0.
\end{align*}
We will take $\lambda>1$. Note that $\|K_t^\lambda\|_1=\|K_1^\lambda\|_1=c_d.$ Hence the normalized function $c_d^{-1}K_t^\lambda$ is a bounded approximate identity. Using this kernel we define two components by
\begin{align*}
	\overrightarrow{G}_{\lambda,f}^*(x)&=\left[ \int_0^\infty t \cdot K_t^\lambda*\Gamma_\alpha(f_t,f_t)(x) \, dt\right]^{1/2} , \\
	\disp {G}_{\lambda,f}^{*,\uparrow}(x)&=\left[ \int_0^\infty t \cdot K_t^\lambda*(\frac{\partial}{\partial t}f_t(\cdot))^2(x) \, dt\right]^{1/2}
\end{align*}
and the $G^*$ functional is 
\begin{align*}
	\disp {G}_{\lambda,f}^{*}(x)= \left[\left[\overrightarrow{G}_{\lambda,f}^{*}(x)\right]^2+ \left[{G}_{\lambda,f}^{*,\uparrow}(x)\right]^2\right]^{1/2}.
\end{align*}

\section{Singular Integral Operators and Boundedness Results}
As we can see in definitions of the operators, we mostly restrict our domain of integration to a parabolic-like domain in the upper half-space. By taking the scaling factor into account, we focus on the set $\{(y,t)\in\Rd\times\Rp:|y-x|<t^{2/\alpha}\}$ with vertex at $x\in \Rd$. Our first observation is that the growth of an extension function is controlled by the Hardy-Littlewood maximal function $\mathcal{M}(\cdot)$, where the Hardy-Littlewood maximal function is given by
$$\mathcal{M}(f)(x)=\sup_{r>0}\frac{1}{|B(0,1)|\cdot r^d}\int_{|y|<r}|f(x-y)|\, dy.$$
To see this, we define 
\begin{align*}
	N_\alpha^f(x):=\sup\{|f_t(y)|: t>0, |x-y|<t^{2/\alpha}\}.
\end{align*}
The classical version of this function is sometimes referred to as the (non-tangential) maximal function. (See \cite[Chapter II]{Stein2}.) In that case, the growth of this function is studied at a single point $x\in\Rd$. In our setup, we should consider the terms corresponding to jumps of the horizontal process. However, we still need to restrict our function to small jumps so that comparison of the points at any given ``height" is possible by Harnack's inequality. For this purpose, the domain is considered to be the parabolic-like region given above.

\begin{lem}
Let $p>1$ and  $f\in L^p(\Rd)$. Then 
\begin{itemize}
	\item[i.] $N_\alpha^f(x)\leq c\, \mathcal{M}(f)(x),\quad x\in\Rd,$ \\
	\item[ii.] $N_\alpha^f \in L^p(\Rd)$ and $\|N_\alpha^f\|_p\leq c\, \|f\|_p.$
\end{itemize}
\end{lem}
\begin{proof} It is enough to consider positive functions to prove the first statement. If $f$ is not positive, then we can consider the decomposition $f=f^+-f^-$, where $f^+,f^-\geq 0$. Then we can use linearity of the semi-group $Q_t$, the inequalities
$$N_\alpha^f\leq N_\alpha^{f^+}+N_\alpha^{f^-}\mbox{ and } \mathcal{M}(f^+)+\mathcal{M}(f^-)\leq 2 \mathcal{M}(f) $$
and the fact that both $Q_tf^+$ and $Q_tf^-$ are positive harmonic to prove the result for $f$. Hence we can reduce our problem to positive functions.
So suppose $f>0$. Then for a fixed $t>0$ and $y\in B(x,t^{2/\alpha})$, Theorem \ref{harnack_inequality} applied several times implies that $f_t(y)\leq c\,f_t(x)$. Here we should emphasize that the constant $c$ does not depend on the variable $t$, since these balls scale as $t$ varies and so the same number of application of the Harnack inequality suffices at each $t$ for fixed $x$. 

Moreover, $f_t(x)=f*q_t(x)$ where $q_t$ is radially decreasing and its $L^1$-norm equals one. To see this we note that  the transition density $p(s,x,0)$ is obtained from the characteristic function $e^{-s|x|^\alpha}$ by the inverse Fourier transform. Hence we can write $p(s,x,0)$ as in equation (\ref{dens_rep}). Thus $p(s,x,0)$ is radially decreasing in the variable $x$ and so is $q_t(x)$. Then $f_t(x) \leq c\, \mathcal{M}(f)(x)$ for any $t>0$ \cite[section 2.1]{grafakos}  and $N_\alpha^f(x)\leq c\, \mathcal{M}(f)(x)$. Finally, using the fact $$\|\mathcal{M}(f)\|_p\leq c\, \|f\|_p,\qquad p>1,$$ one can obtain the result.  
\end{proof}

Before we study the Area functional, we define an auxiliary operator $L_f^*$. This operator is in a close relation with $\overrightarrow{G}^*_{\lambda , f}$ for a particular value of $\lambda$ and hence it provides an intermediate step to prove boundedness of the Area functional. Moreover, the classic version $L^*_f$ is used to give a probabilistic proof of boundedness of Littlewood-Paley function. 

For a given $f\in L^p(\Rd)$, we define this operator as 
\begin{align*}
	\disp L_f^*(x)=\left[ \int_0^\infty t\cdot Q_t\Gamma_\alpha(f_t,f_t)(x) \, dt\right]^{1/2},
\end{align*}
where $\Gamma_\alpha$ is as in (\ref{gamma_alpha}).
This operator is bounded on $L^p(\Rd)$ whenever $p>2$.
\begin{thm}\label{l-star}
Let $p>2$ and $f\in L^p(\Rd)$. Then we have $$\|L_f^*\|_p\leq c \, \|f\|_p.$$
\end{thm}

\begin{proof}
	Let $f\in L^p(\Rd)$, $r=2p$ and $q$ be the conjugate of $r$, that is, $1/r+1/q=1$. Let $h$ be a continuously differentiable function with compact support. Then
		\begin{align*}
			\lefteqn{\E^{(x,a)}\left[ \int_0^{T_0} \Gamma_\alpha (f_{Z_s},f_{Z_s})(Y_s)ds \cdot h(X_{T_0})	\right]} \\
			&=\int_0^\infty \E^{(x,a)}\left[   \E^{(x,a)}\left[ \ind_{\{s<T_0\}}\Gamma_\alpha(f_{Z_s},f_{Z_s})(Y_s) h(X_{T_0})\, |  \mathcal{F}_s \right] \right] \, ds\\
			&= \E^{(x,a)}\left[  \int_0^\infty \ind_{\{s<T_0\}}\Gamma_\alpha(f_{Z_s},f_{Z_s})(Y_s)  \E^{(x,a)}\left[h(X_{T_0})\, |  \mathcal{F}_s \right] ds\right]\\
			&= \E^{(x,a)}\left[  \int_0^\infty \ind_{\{s<T_0\}}\Gamma_\alpha(f_{Z_s},f_{Z_s})(Y_s)  \E^{X_s}\left[h(X_{T_0})\,  \right] ds\right],
		\end{align*}
	by Markov property. Then using invariance of the semi-group $P_t$ under the Lebesgue measure (equation (\ref{invariance})) and the vertical Green function (equation (\ref{greens_function})), respectively, we obtain
		\begin{align*}
			\lefteqn{\E^{m_a}\left[ \int_0^{T_0} \Gamma_\alpha (f_{Z_s},f_{Z_s})(Y_s)ds \cdot h(X_{T_0})	\right]} \\
			&=\int_{\Rd} \E^{a}\left[  \int_0^{T_0} \Gamma_\alpha(f_{Z_s},f_{Z_s})(x) \cdot  \E^{(x,Z_s)}\left[h(X_{T_0})\,  \right] ds\right] dx\\
			&=\int_{\Rd}  \int_0^{\infty} (a\wedge t) \Gamma_\alpha(f_{t},f_{t})(x) \cdot  \E^{(x,t)}\left[h(X_{T_0})\,  \right] dt\, dx.
		\end{align*}	
	Now if we take the limit as $a\rightarrow \infty$, the last expression above approaches
		\begin{align*}
			\int_{\Rd}  \int_0^{\infty} t \cdot \Gamma_\alpha(f_{t},f_{t})(x) \cdot  h_t(x)\,  dt\, dx.
		\end{align*}	
	By the symmetry of the kernel $q_t(\cdot)$, this limit equals
		\begin{align*}
			\int_{\Rd}  \int_0^{\infty} t \cdot \Gamma_\alpha(f_{t},f_{t})(x) \cdot  h*q_t(x)\,  dt\, dx&=\int_{\Rd}  \int_0^{\infty} t \cdot \Gamma_\alpha(f_{t},f_{t})*q_t(x) \cdot  h(x)\,  dt\, dx\\
				&=\int_{\Rd} h(x) \left( L_f^*(x)\right)^2 dx.
		\end{align*}	
	Next, using the H\"older inequality with exponents $q$ and $r$,
		\begin{align*}
			\lefteqn{\E^{m_a}\left[ \int_0^{T_0} \Gamma_\alpha (f_{Z_s},f_{Z_s})(Y_s)ds \cdot h(X_{T_0})	\right]} \\
			&\qquad \leq \left( \E^{m_a} | h(X_{T_0})|^q\right)^{1/q} \, \left( \E^{m_a}\left[ \int_0^{T_0} \Gamma_\alpha (f_{Z_s},f_{Z_s})(Y_s)ds\right]^r \right)^{1/r}.
		\end{align*}	
	Now denote the martingale $f(X_{t\wedge T_0})$ by $M_t^f$. By \cite[p. 158]{meyer_long} or \cite[Section 2]{karli},
		\begin{align*}
			\E^{m_a}\left[ \int_0^{T_0} \Gamma_\alpha (f_{Z_s},f_{Z_s})(Y_s)ds\right]^r \leq c\, \E^{m_a}\left[ \int_0^{T_0} g(Y_s,Z_s)ds\right]^r \leq c\, \Ema\left[ \langle M^f\rangle_{T_0}\right]^r,
		\end{align*}	
	where the function $g(y,t)$ is defined by 
		\begin{align}\label{g_func}
			g(x,t)=\int_{\Rd} \left[f_t(x+h)-f_t(x)\right]^2\frac{dh}{|h|^{d+\alpha}}+\left[\frac{\partial}{\partial t} f(x,t)\right]^2.
		\end{align}
	By Burkholder-Gundy-Davis inequality, the last term is bounded by a constant multiple of $\Ema\left[  \sup_{s\leq T_0} |M^f_{s}|\right]^{2r}$ which is bounded by $c\, \Ema\left|  M^f_{T_0}\right|^{2r}$ by Doob's inequality. Hence
		\begin{align*}
			 \lefteqn{\lim_{a\rightarrow \infty} \E^{m_a}\left[ \int_0^{T_0} \Gamma_\alpha (f_{Z_s},f_{Z_s})(Y_s)ds \cdot h(X_{T_0})	\right]} \\
			&\quad \leq c \lim_{a\rightarrow \infty}  \left( \E^{m_a} | h(X_{T_0})|^q\right)^{1/q}  \left( \E^{m_a} | f(X_{T_0})|^{2r}\right)^{1/r} \leq c\, \|h\|_q\, \|f\|_{2r}^2.
		\end{align*}
	Using the first part,
		\begin{align*}
			\int_{\Rd} h(x) \, \left( L_f^*(x)\right)^2 dx \leq c\,  \|h\|_q\, \|f\|_{2r}^2.
		\end{align*}	
	Finally, if we take supremum  over all such $h$ with $\|h\|_q\leq 1$, then 
		$$\left[ \int_{\Rd} \left(L_f^*(x)\right)^{2r} dx\right]^{1/r}\leq c\, \|f\|_{2r}^2,$$
	which gives the result if we replace $r$ with $p/2$. 
\end{proof}

 Now, if we consider $\lambda_0=(2d+\alpha)/(2d)$ then we can see the relation between two functionals $L^*_f$ and $\overrightarrow{G}^*_{\lambda_0,f}$. Hence we can show boundedness of the area functional $A_f$.

\begin{thm} Suppose $p>2$ and $f\in L^p(\Rd)$. Then we have
	\begin{itemize}
		\item[i.] for $\lambda>0$, $A_f \leq c_\lambda \, \overrightarrow{G}^*_{\lambda,f} $.\\
		\item[ii.] If $\lambda_0=(2d+\alpha)/(2d)$, then $$\|\overrightarrow{G}^*_{\lambda_0,f}\|_p \leq c\, \|f\|_p.$$ 
		\item[iii.] $\|A_f\|_p\leq c \|f\|_p$.
	\end{itemize}
\end{thm}
\begin{proof}
	Part (i.) is easy when we observe $$\left[\frac{t^{2/\alpha}}{t^{2/\alpha}+|y|}\right]^{\lambda d}\geq 2^{-\lambda d}$$ for $|y|<t^{2/\alpha}$.
	Part (iii.) is a corollary of (i.) and (ii.). So it is enough to prove (ii.).
	First we recall that 
	$$K_t^{\lambda_0}(x)=\frac{t}{(t^{2/\alpha}+|x|)^{d+\frac{\alpha}{2}}}=t^{-2d/\alpha} \left(\frac{1}{1+\frac{|x|}{t^{2/\alpha}}}\right)^{d+\frac{\alpha}{2}}.$$
	We also know that $q_t(x)$ is comparable to $$t^{-2d/\alpha} \wedge \frac{t}{|x|^{d+\frac{\alpha}{2}}}=t^{-2d/\alpha} \left(1\wedge \frac{1}{\left(\frac{|x|}{t^{2/\alpha}}\right)^{d+\frac{\alpha}{2}}}\right),$$ by (\ref{qt_estimate}).
	Hence $q_t$ is comparable to $K_t^{\lambda_0}$ and we have
	\begin{align*}
		K_t^{\lambda_0}\leq c\, q_t(x).
	\end{align*}
	This leads to 
	\begin{align*}
		\overrightarrow{G}^*_{\lambda_0,f}(x)\leq c L_f^*(x).
	\end{align*}
	Then the result follows from Theorem \ref{l-star}. 
\end{proof}

The result of the previous theorem is not restricted to the horizontal component with parameter $\lambda_0$. We can generalise this result to the case including the vertical component  and any parameter $\lambda>1$.
\begin{thm}\label{g_star}
	If $\lambda>1, p\geq 2$ and $f\in L^p(\Rd)$ then $$\|{G}^*_{\lambda,f}\|_p \leq c\, \|f\|_p.$$
\end{thm}
\begin{proof}
	Let us denote by $g_\alpha(y,t)$ the function 
	$$\Gamma_\alpha(f_t,f_t)(y) + \left(\frac{\partial}{\partial t}f(y,t)\right)^2.$$
	Assume $h\in \mathcal{C}_K^1(\Rd)$. Then by the symmetry of $K_t^\lambda(x)$ in $x$,
	\begin{align*}
		\int_{\Rd} h(x) \, (G^*_{\lambda,f}(x))^2dx&=\int_0^\infty t \int_{\Rd} h(x) \int_{\Rd} K_t^\lambda(x-y)\,g_\alpha(y,t)\, dy\, dx\, dt\\
			&=\int_0^\infty t \int_{\Rd} g_\alpha(y,t) \cdot h*K_t^\lambda(y) \, dy\, dt.
	\end{align*}
	Since $K_t^\lambda$ is radially decreasing and integrable, $h*K_t^\lambda(y)\leq c\, \mathcal{M}(h)(y)$. Hence 
	\begin{align}\label{g_max_ineq}
		\int_{\Rd} h(x) \, (G^*_{\lambda,f}(x))^2dx&\leq c\int_{\Rd} \mathcal{M}(h)(x) \, (G_{f,\alpha}(x))^2 dx.
	\end{align}
	For $p=2$, it is enough to consider $h=1$. Then by parts (ii.) and (iii.) of Theorem \ref{LP_thm},
	$$\|{G}^*_{\lambda,f}\|_2\leq c \|{G}_{f,\alpha}\|_2\leq c\, \|f\|_2.$$
	Now suppose $p>2$. We take $r=p/2$ and $q>0$ such that $1/r+1/q=1$. Using H\"older's inequality in (\ref{g_max_ineq}),
	\begin{align*}
		\int_{\Rd} h(x) \, (G^*_{\lambda,f}(x))^2dx&\leq c \left[ \int_{\Rd} (\mathcal{M}(h)(x))^q\,dx \right]^{1/q} \cdot \left[ \int_{\Rd} (G_{f,\alpha}(x))^{2r}\,dx \right]^{1/r} \\
			& \leq c\, \|h\|_q \|G_{f,\alpha}\|_p^2.
	\end{align*}
	If we take supremum over all such $h$ with $\|h\|_q\leq 1$, then we obtain
	\begin{align*}
		\|G^*_{\lambda,f}\|_p^2=\|(G^*_{\lambda,f})^2\|_r\leq c\, \|G_{f,\alpha}\|_p^2.
	\end{align*} 
	Finally, using the boundedness of the operator $G_{f,\alpha}$ when $p>2$ (Theorem \ref{LP_thm}), we prove the desired result. 
\end{proof}

In the final part of the paper, we discuss an application of the previous Theorem. We will provide a result on the boundedness of singular integrals which is a generalization of Theorem \ref{classic_case}. We show that the result holds under a weaker condition on the tail of the kernel. For this purpose, we impose a boundedness condition in terms of the semi-group $Q_t$. 
\begin{thm}\label{mult_1}
	Let $p>1$. Suppose $T$ is a convolution operator on $L^p(\Rd)$ with kernel $\kappa$, that is, $Tf(x)=f*\kappa(x)$. Suppose further that there exists $\lambda>1$ such that 
	\begin{align} \label{assump_1}
		|\partial_t Q_t\kappa(x)|\leq c\, t^{-1-2d/\alpha}\left( \frac{t^{2/\alpha}}{t^{2/\alpha}+|x|} \right)^{\lambda d}=c\, t^{-1}K_t^\lambda(x).
	\end{align}
	 Then for $f\in C_K^1$ (, that is,  $f\in C^1$ with compact support) $$\|Tf\|_p\leq c \, \|f\|_p.$$
\end{thm}
The condition  (\ref{assump_1}) above may not seem very useful in terms of applications when we consider its current form. Hence we will provide a sufficient and a more useful condition later in Theorem \ref{main_thm} below.
\begin{proof}
	First suppose $p>2$. We note that by the semi-group property, we have $Q_t=Q_{t/2}Q_{t/2}$ and $q_t=q_{t/2}*q_{t/2}$, which leads to $\partial_tq_t=2q_{t/2}*\partial_tq_{t/2}$. Next,  we observe that $\partial_t Q_{t}Tf(x)=2Q_{t/2}T(\partial_tQ_{t/2}f)(x)$, since their Fourier transforms are equal , that is,
	$$\widehat{\left[ 2Q_{t/2}T(\partial_tQ_{t/2}f)\right]}=2\widehat{q_{t/2}}\,\widehat{\kappa}\,\widehat{(\partial_tq_{t/2})}\widehat{f}=\widehat{(2q_{t/2}*\partial_tq_{t/2})\,}\widehat{\kappa}\,\widehat{f}=\widehat{\partial_tq_{t}}\,\widehat{\kappa}\,\widehat{f}=\widehat{\partial_t Q_{t}Tf}.$$
	Then
	\begin{align*}
		\left(G^\uparrow_{Tf}(x)\right)^2&=\int_0^\infty t\, \left| \partial_t Q_tTf(x) \right|^2 \, dt=4\int_0^\infty t\, \left|Q_{t/2}T(\partial_tQ_{t/2}f)(x) \right|^2 \, dt.
	\end{align*}
	Using our assumption (\ref{assump_1}), we can see that $Q_{t/2}T(\partial_tQ_{t/2}f)=(1/2)\partial_t Q_{t}Tf(x)\rightarrow 0$ as $t\rightarrow \infty$. Hence the last line above equals
	\begin{align*}
		4\int_0^\infty t\, \left|\int_t^\infty \frac{s}{s}\, \partial_s Q_{s/2}T(\partial_sQ_{s/2}f)(x)ds \right|^2 \, dt.
	\end{align*}
	If we apply the Cauchy-Schwartz inequality first, and then change the order of the integrals, we get
	\begin{align*}
		\left(G^\uparrow_{Tf}(x)\right)^2&\leq c \int_0^\infty t \left[ \int_t^\infty s^{-2}ds\right]\cdot\left[ \int_t^\infty s^2\, (\partial_s Q_{s/2}T(\partial_sQ_{s/2}f)(x))^2 ds \right]\, dt\\
			&=c\, \int_0^\infty   \int_t^\infty s^2\, (\partial_s Q_{s/2}T(\partial_sQ_{s/2}f)(x))^2 ds  dt\\
			&=c\, \int_0^\infty   s^3\, (\partial_s Q_{s/2}T(\partial_sQ_{s/2}f)(x))^2 ds  .\\
	\end{align*}
	Using the bound in (\ref{assump_1}) and Jensen's inequality,
	\begin{align*}
		\left(G^\uparrow_{Tf}(x)\right)^2&\leq c\, \int_0^\infty   s^3\, \left[(s^{-1}K_{s/2}^{\lambda})*(\partial_sQ_{s/2}f)(x))\right]^2 ds  \\
				& \leq c\, \int_0^\infty   s\, K_{s/2}^{\lambda}*(\partial_sQ_{s/2}f)^2(x) ds \\
				&\leq c \left( G^*_{\lambda,f}(x)\right)^2.
	\end{align*}
	Hence for $p>2$,
	\begin{align*}
		\|Tf\|_p\leq c\|G^\uparrow_{Tf}\|_p\leq c\, \|G^*_{\lambda,f}\|_p\leq c\, \|f\|_p,
	\end{align*}
	by Lemma \ref{LP_lemma} and Theorem \ref{g_star}.
	
	For $p\in(1,2)$ we use a duality argument. For this purpose let $q$ be such that $1/p+1/q=1$. First we observe that if $\kappa^*(x)=\kappa(-x)$ and $T^*$ is the convolution operator corresponding to $\kappa^*$, then the condition (\ref{assump_1}) holds for $\kappa^*$. Then for $h\in L^q(\Rd)$
	\begin{align*}
		\left|\int_{\Rd} h(x) \, Tf(x)\, dx\right|=\left|\int_{\Rd} T^*h(x) \, f(x)\, dx\right|\leq c \|T^*h\|_q\, \|f\|_p \leq \|h\|_q\, \|f\|_p
	\end{align*}
	by the first part of the proof. Finally, if we take supremum over all such $h$ with $\|h\|_q\leq 1$, the result follows. 
\end{proof}

Before any further discussion, we recall the definition of the measure 
$$\mu_t(ds)=\frac{t}{2\sqrt{\pi}}e^{-t^2/4s}s^{-3/2}ds$$
and show the following estimates.

\begin{lem}\label{tech_lemma_1}
For $M>0$, we have
\begin{itemize}
	\item[i.] $\disp \int_0^M |s-1/2|\, \mu_1(ds)\leq \frac{1}{\sqrt{\pi}}\, M^{1/2}$, \\
	\item[ii.] $\disp \int_M^\infty |1-1/(2s)|\, \mu_1(ds)\leq \frac{1}{\sqrt{\pi}}\, M^{-1/2}$.
\end{itemize}
\end{lem}
\begin{proof}
 (i.) We note that 
	\begin{align*}
		|s-1/2|\,  s^{-1/2} e^{-1/(4s)} \leq 1
	\end{align*}
for $s>0$. Hence the result follows.

 (ii.) Similarly, we also have 
	\begin{align*}
		|s-1/2|\,  e^{-1/(4s)} \leq 1,
	\end{align*}
	which results in the desired inequality.
\end{proof}

\begin{lem}\label{estimate_1}
Suppose $\disp\psi(x)=\left(\partial_tq_t(x)\right)_{t=1}=\int_0^\infty p(s,x,0)\left(1-\frac{1}{2s}\right)\mu_1(ds)$. Then for some positive constants $c, c_1, c_2$ we have
\begin{itemize}
	\item[i.] $\disp \left| \psi(x)\right|\leq c_1\, \left(1 \wedge |x|^{-d-\frac{\alpha}{2}}\right)\leq c_2 \,q_1(x)$ and \\
	\item[ii.] $\disp \left| \partial_{x_i} \psi(x)\right|\leq c\,\left(1 \wedge |x|^{-d-1-\alpha/2}\right)$,  $i=1,...,d$.
\end{itemize}

\end{lem}
\begin{proof}(i.) First note that 
$$|\psi(x)|\leq c \int_0^\infty p(s,x,0)\, |1-\frac{1}{2s}|\, s^{-3/2}e^{-1/(4s)}ds\leq c\int_0^\infty  |1-\frac{1}{2s}|\, s^{-\frac{3}{2}-\frac{d}{\alpha}}e^{-1/(4s)}ds<\infty$$
by the estimate (\ref{sas_estimate}) on the density $p(s,x,0)$. Using the same estimate once again we obtain
	\begin{align*}
		|\psi(x)|&\leq c \int_0^{|x|^\alpha} \frac{s}{|x|^{d+\alpha}}| 1-\frac{1}{2s} |\, \mu_1(ds)+c\, \int_{|x|^\alpha}^\infty s^{-d/\alpha} | 1-\frac{1}{2s} |\, \mu_1(ds)\\
				&\leq  \frac{c}{|x|^{d+\alpha}} \int_0^{|x|^\alpha}| s-\frac{1}{2} |\, \mu_1(ds)+ \frac{c}{|x|^{d}}  \int_{|x|^\alpha}^\infty  | 1-\frac{1}{2s} |\, \mu_1(ds).
	\end{align*}
By Lemma \ref{tech_lemma_1},
	$$\disp \left| \psi(x)\right|\leq c_1\, \left(1 \wedge |x|^{-d-\alpha/2}\right).$$
The second inequality follows from the estimate (\ref{qt_estimate}) on $q_1(x)$.

(ii.) Similarly, using the bound on $\partial_{x_i} p(s,x,0)$ (Lemma \ref{lemma_der_of_dens}), we obtain
	$$|\partial_{x_i}\psi(x)|\leq c \int_0^\infty | \partial_{x_i}p(s,x,0)| |1-\frac{1}{2s}|\, \mu_1(ds)\leq c\int_0^\infty  |1-\frac{1}{2s}|\, s^{-\frac{3}{2}-\frac{d+1}{\alpha}}e^{-1/(4s)}ds<\infty$$
and
	\begin{align*}
		|\partial_{x_i}\psi(x)|&=\left| \int_0^\infty  \partial_{x_i}p(s,x,0) (1-\frac{1}{2s})\, \mu_1(ds)	\right|\\
				&\leq c \int_0^{|x|^\alpha} \frac{s}{|x|^{d+1+\alpha}}| 1-\frac{1}{2s} |\, \mu_1(ds)+c\, \int_{|x|^\alpha}^\infty s^{-(d+1)/\alpha} | 1-\frac{1}{2s} |\, \mu_1(ds)\\
				&\leq  \frac{c}{|x|^{d+1+\alpha}} \int_0^{|x|^\alpha}| s-\frac{1}{2} |\, \mu_1(ds)+ \frac{c}{|x|^{d+1}}  \int_{|x|^\alpha}^\infty  | 1-\frac{1}{2s} |\, \mu_1(ds).
	\end{align*}
When we use Lemma \ref{tech_lemma_1}, we obtain the desired result. 
\end{proof}

In the previous theorem, we stated a boundedness condition on the kernel of convolution operator by means of the action of the semi-group $Q_t$. In order for this condition to be more useful,  we want to state an application in some purely analytic language. We provide two conditions in Theorem \ref{main_thm}, under which the condition (\ref{assump_1})  of Theorem \ref{mult_1} holds and hence the result follows.

\begin{thm}\label{main_thm}
 Suppose $\alpha\in (1,2)$, $\kappa:\Rd\rightarrow \mathbb{R}$ is a function with the cancelation property (\ref{cancelation}) such that 
\begin{itemize}
	\item[i.] $\disp |\kappa(x)|\leq \frac{c}{|x|^d}\ind_{\{|x|\leq 1\}}+\frac{c}{|x|^{d-1+\alpha/2}}\ind_{\{|x|> 1\}}$,\\
	\item[ii.] $\disp |\nabla\kappa(x)|\leq \frac{c}{|x|^{d+1}}\ind_{\{|x|\leq 1\}}+\frac{c}{|x|^{d+\alpha/2}}\ind_{\{|x|> 1\}}$.
\end{itemize}
Suppose $T$ is a convolution operator with kernel $\kappa$. Then for $f\in \mathcal{C}_K^1$ and $p>1$ we have
$$\|Tf\|_p \leq c\, \|f\|_p.$$
\end{thm}

\begin{proof}
	First let $\phi$ be a smooth function on $\mathbb{R}$ such that $\phi(r)=1$ whenever $|r|\leq 1$ and  $\phi(r)=0$ whenever $|r|>2$. Now let $$\kappa_1(x)=\kappa(x) \phi(|x|^2),\quad \kappa_2(x)=\kappa(x) (1-\phi(|x|^2))$$ and $$T_1f=f*\kappa_1, \quad  T_2f=f*\kappa_2.$$ Then $Tf=T_1f+T_2f$. By the classical case (Theorem \ref{classic_case}), $\|T_1f\|_p\leq c\, \|f\|_p$. So without loss of generality we may assume $T=T_2$ and $\kappa=\kappa_2$ and ignore the indices. As before, set $\disp\psi(x)=\left(\partial_tq_t(x)\right)_{t=1}$. By scaling, we have $\partial_tq_t(x)=t^{-1-2d/\alpha}\psi(x/t^{2/\alpha})$. So by Theorem \ref{mult_1} and scaling, it is enough to show that 
	\begin{align*}
		\left|  \left( \partial_tQ_t\kappa(x) \right)_{t=1} \right| \leq \frac{c}{\left(1+|x|\right)^{\lambda d}}
	\end{align*}
	for some $\lambda>1$. Here we will take $\lambda=1+(\alpha-1)/(2d)$.\\
	
	First assume $|x|\leq 1$. Then
	\begin{align}\label{main_int}
		 \left( \partial_tQ_t\kappa(x) \right)_{t=1}&=\int_{|y|>1} \kappa(y) \psi(x-y)\, dy
	\end{align}
	and we have by Lemma \ref{estimate_1} (i),
	\begin{align*}
		\left| \left( \partial_tQ_t\kappa(x) \right)_{t=1}\right|&\leq \int_{|y|>1}|\kappa(y)|\,|\psi(x-y)|dy \\
												&\leq c  \int_{|y|>1}|\kappa(y)| \, q_1(x-y)dy. \\
	\end{align*}
	Then the assumption (i) on $\kappa(\cdot)$ in the hypothesis gives us that 
	$$|\kappa(y)|\leq \frac{c}{|y|^{d-1+(\alpha/2)}}\leq c$$ whenever $|y|>1.$ Thus
	\begin{align*}
		 \int_{|y|>1}|\kappa(y)| \, q_1(x-y)dy &\leq c \int_{|y|>1}q_1(x-y) dy\leq c \int_{\Rd} q_1(x-y)dy= c, \\
	\end{align*}
	since $q_1$ is a probability kernel.
	Hence
	\begin{align}\label{leq1}
		\left| \left( \partial_tQ_t\kappa(x) \right)_{t=1}\right|&\leq c	\leq \frac{c}{\left(1+|x|\right)^{\lambda d}}.
	\end{align}
	Now assume $|x|>1$. Consider three subsets of $\Rd$: $\mathcal{D}_1=\{y\in\Rd:|y|<|x|/2\}$,  $\mathcal{D}_2=\{y\in\Rd:|y-x|<|x|/2\}$ and  $\mathcal{D}_3=\Rd-(\mathcal{D}_1\cup \mathcal{D}_2)$. Now split the integral (\ref{main_int}) into three integrals with respect to these subsets.
	\begin{align}
		\int_{|y|>1} \kappa(y) \psi(x-y)\, dy&=\int_{\mathcal{D}_1}+\int_{\mathcal{D}_2}+\int_{\mathcal{D}_3}:=I_1+I_2+I_3.\nonumber
	\end{align}
	Since $\kappa$ satisfies the cancelation condition (\ref{cancelation}), 
	\begin{align*}
		|I_1|&=\left| \int_{\mathcal{D}_1}  \kappa(y)\,(\psi(x-y)-\psi(x)) dy \right|\\
			&\leq \sup_{|z-x|<|x|/2} \left|\nabla\psi(z)\right| \int_{\mathcal{D}_1}  \frac{c}{|y|^{d-1+\alpha/2}}\,|y| dy \\
			&\leq c \, |x|^{2-\alpha/2} \sup_{|z-x|<|x|/2} \left| \nabla\psi(z)\right| .
	\end{align*}
	By Lemma \ref{estimate_1}, the gradient above is bounded by $c |x|^{-d-1-\alpha/2}$. This gives the inequality
	\begin{align*}
		|I_1|\leq \frac{c}{|x|^{d-1+\alpha}}\leq \frac{c}{|x|^{\lambda d}}.
	\end{align*}
	For $I_3$, we note that $c|y|\leq |x-y| \leq c'|y|$ whenever $y\in \mathcal{D}_3$. So using Lemma \ref{estimate_1} again,
	\begin{align*}
		|I_3|\leq c \, \int_{\mathcal{D}_3} \frac{1}{|x-y|^{d+\alpha/2}} \frac{1}{|y|^{d-1+\alpha/2}}dy\leq c  \, \int_{|y|\geq |x|/2}|y|^{-2d+1-\alpha}dy\leq \frac{c}{|x|^{\lambda d}}.
	\end{align*}
	For the last part, namely $I_2$, we use our assumption on $\nabla\kappa$. By a change of variables, we have
	\begin{align*}
		|I_2|&=\left|\int_{|y-x|<|x|/2}\psi(x-y)\, \kappa(y)\, dy\right|\\
			&=\left|\int_{|y|<|x|/2}\psi(y)\, \kappa(x-y)\, dy\right|\\
			& \leq \int_{|y|<|x|/2}|\psi(y)|\,| \kappa(x-y)-\kappa(x)|\, dy+|\kappa(x)| \left|\int_{|y|<|x|/2}\psi(y)\,  dy\right|.
	\end{align*}
	First observe that 
	\begin{align*}
		\int_{\Rd} \psi(y)dy=\int_0^\infty\int_{\Rd} p(s,y,0)dy\, (1-\frac{1}{2s})\mu_1(ds)=0.
	\end{align*}
	Hence
	\begin{align*}
		\left| \int_{|y|<|x|/2} \psi(y) dy \right|=\left| \int_{|y|\geq |x|/2} \psi(y) dy \right|\leq c  \int_{|y| \geq |x|/2} |y|^{-d-\alpha/2} dy \leq \frac{c}{|x|^{\alpha/2}}.
	\end{align*}
	We also note that if $|y|<|x|/2$ then using the upper-bound on $|\nabla\kappa|$, we obtain
	\begin{align*}
		\left| \kappa(x-y)-\kappa(x) \right| \leq c\, \frac{|y|}{|x|^{d+\alpha/2}}\leq c\,  \frac{|y|^{1/2}}{|x|^{d+(\alpha-1)/2}} = c\,  \frac{|y|^{1/2}}{|x|^{\lambda d}}.  
	\end{align*}
	Then
	\begin{align}\label{est_I_2}
		|I_2|&\leq \frac{c}{|x|^{\lambda d}}\int_{\Rd} \left| \psi(y)\right|\cdot |y|^{1/2}dy+\frac{c}{|x|^{d-1+\alpha}}.
	\end{align}
	If we show that the integral in (\ref{est_I_2}) is bounded by a constant, then we have
	\begin{align*}
		|I_2|&\leq \frac{c}{|x|^{\lambda d}}.
	\end{align*}
	To show the boundedness of the integral in (\ref{est_I_2}), we consider the cases $|y|<1$ and $|y|\geq 1$. Note that
	\begin{align*}
		\int_{\Rd} \left| \psi(y)\right|\cdot |y|^{1/2}dy\leq \int_{|y|<1} \left| \psi(y)\right|dy+\int_{|y|\geq 1} \frac{c}{|y|^{d+\alpha/2}}\cdot |y|^{1/2}dy
	\end{align*}
	by Lemma \ref{estimate_1}. The second term is convergent since $\alpha>1$. The first term is bounded by 
	\begin{align*}
		\int_{|y|< 1 }\int_0^\infty p(s,y,0) |1-\frac{1}{2s}|\, \mu_1(ds)\,dy&\leq \int_0^\infty \int_{\Rd} p(s,y,0)\, dy |1-\frac{1}{2s}|\, \mu_1(ds)\\
				&\leq \int_0^{1/2} \frac{1}{2s} \, \mu_1(ds)+\int_{1/2}^\infty \mu_1(ds)\\
				&\leq c.
	\end{align*}
	Hence 
	\begin{align}\label{geq1}
		\left| \left(\partial Q_t\kappa(x)\right)_{t=1}\right|\leq |I_1|+|I_2|+|I_3|\leq \frac{c}{|x|^{\lambda d}}\leq\frac{c}{(1+|x|)^{\lambda d}} 
	\end{align}
	whenever $|x|>1$.
	Finally, inequalities (\ref{leq1}) and (\ref{geq1}) and Theorem \ref{mult_1} imply that 
	$$\|Tf\|_p\leq c\, \|f\|_p$$
	for $p>1$ and $f\in\mathcal{C}_K^1$, which finishes the proof. 

\end{proof}

\subsection*{Acknowledgements}
This research project is supported by the BAP grant, numbered 14B103, at the I\c{s}\i k University, Istanbul, Turkey.
We also would like to thank to our anonymous referee for his/her comments.


\begin{thebibliography}{HD}

%% Use the widest label as the parameter.
%% Reference items can be numbered or have labels of your choice, as below.

%% In IMPAN journals, only the title is italicized; boldface is not used.
%% Our software will add links to many articles; for this, enclosing volume numbers in { } is helpful
%% Do not give the issue number unless the issues are paginated separately.



%%%%%%% To ease editing, use normal size:

\normalsize
\baselineskip=17pt

%%%%%%%%%%%%%%%
\bibitem{bass_sing_int}
   R.F. Bass, \emph{A Probabilistic Approach to the Boundedness of Singular Integral Operators}, S\'eminaire de probabilit\'es 24  (1990), 15-40.

\bibitem{bass_prob_tech}
   R.F. Bass, \emph{Probabilistic Techniques in Analysis}, Springer-Verlag, NewYork, 1995.

\bibitem{bass_zqchen}
   R.F. Bass, Z.-Q. Chen, \emph{Systems of Equations Driven by Stable Processes}, Probab. Theory Related Fields 134 (2006),  175-214.

\bibitem{banuelos}
  R. Ba$\tilde{\mbox{n}}$uelos, S.Y. Yolcu, \emph{Heat Trace of Non-local Operators}, J. London Math. Soc. (2012).

\bibitem{banuelos2}
R. Ba$\tilde{\mbox{n}}$uelos, B. Davis, \emph{Donald Burkholder's work in martingales and analysis}. In "Selected Works of Donald L. Burkholder," B. Davis, R. Song, Editors. Springer 2011.

\bibitem{blumenthal}
R.M. Blumenthal, R.K. Getoor, \emph{Some Theorems on Stable Processes}, Trans. Amer. Math. Soc., 95 (1960), 263-273.

\bibitem{BG}
D.L. Burkholder, R.F. Gundy, \emph{Distribution function inequalities for the area integral}, Studia Math. 44 (1972), 527-544.

\bibitem{BGS}
D.L. Burkholder, R. F. Gundy, \emph{M. L. Silverstein: A maximal function characterization of the class Hp}, Trans. Amer. Math. Soc. 157 (1971), 137-153 .

\bibitem{durrett}
R.Durrett, \emph{Brownian Motion and Martingales in Analysis}, Wadsworth, Belmont, CA, 1984.

\bibitem{grafakos}
  L. Grafakos, \emph{Classical and Modern Fourier Analysis}. Prentice Hall, NewJersey, 2004.

\bibitem{gundy}
R.F. Gundy, \emph{Some Topics in Probability and Analysis}. CBMS Regional Conference Series in Mathematics, 70 American Mathematical Society, Providence, RI, (1989).

\bibitem{karli}
  D. Karl\i, \emph{Harnack Inequality and Regularity for a Product of Symmetric Stable Process and Brownian Motion}, Potential Analysis 38 (2013), 95-117, DOI: 10.1007/s11118-011-9265-6 . 

\bibitem{kim}
  I. Kim and K. Kim, \emph{A generalization of the Littlewood-Paley inequality for the fractional Laplacian $(-\Delta)^{\alpha/2}$}; J. Math. Anal. Appl. 388, no. 1 (2012), 175-190.

\bibitem{meyer_long}
  P. A. Meyer, \emph{D\'emonstration probabiliste de certaines in\'egalites de Littlewood-Paley}. S\'eminaire de probabilit\'es (Strasbourg) 10 (1976), 164-174.

\bibitem{meyer_long_2}
  P. A. Meyer, \emph{D\'emonstration probabiliste de certaines in\'egalites de Littlewood-Paley}. Expos\'e IV : semi-groupes de convolution sym\'etriques. S\'eminaire de probabilit\'es (Strasbourg) 10 (1976), 175-183.

\bibitem{sato}
K-I. Sato, \emph{L\'evy Processes and Infinitely Divisible Distributions}, Cambridge University
Press, Cambridge, 1999.

\bibitem{stein}
E.M. Stein, \emph{Singular Integrals and Differentiability  Properties of Functions}. Princeton Univ. Press, Princeton, 1970.

\bibitem{stein3}
E.M. Stein, \emph{The development of square functions in the work of A. Zygmund}, Bull. Amer. Math. Soc. 7 (1982), 359-376.

\bibitem{Stein2}
E.M. Stein, \emph{Topics in harmonic analysis related to the Littlewood-Paley Theory}, Annals of Math. Studies 63, Princeton, 1970.

\bibitem{varopoulos}
N. Varopoulos, \emph{Aspects of probabilistic Littlewood-Paley theory}, Journal of Functional Analysis 38 (1980), 25-60.


\end{thebibliography}
\end{document}